\documentclass[12pt]{amsart}
\usepackage{amsmath, amsthm, amssymb, amsfonts}
\usepackage{enumitem} 
\usepackage[dvipdfm, left=20mm, right=20mm, top=25mm, bottom=20mm]{geometry}

\usepackage{amsmath, amssymb, mathtools}
\usepackage{amsthm}
\usepackage{here}
\usepackage{hyperref}
\usepackage{color}
\usepackage{cleveref}
\usepackage[english]{babel}

\theoremstyle{plain}
\newtheorem{thm}{Theorem}[section]
\crefname{thm}{Theorem}{Theorems}
\newtheorem{lem}[thm]{Lemma}
\crefname{lem}{Lemma}{Lemmas}
\newtheorem{prop}[thm]{Proposition}
\crefname{prop}{Proposition}{Propositions}
\newtheorem{cor}[thm]{Corollary}
\crefname{cor}{Corollary}{Corollaries}

\theoremstyle{definition}

\newtheorem{dfn}[thm]{Definition}
\crefname{dfn}{Definition}{Definitions}
\newtheorem{ex}[thm]{Example}
\crefname{ex}{Example}{Examples}

\theoremstyle{remark}
\newtheorem{rem}[thm]{Remark}
\crefname{rem}{Remark}{Remarks}
\crefname{figure}{Figure}{Figures}
\crefname{section}{Section}{Sections}
\crefname{subsection}{Subsection}{Subsections}

\numberwithin{equation}{section}

\newcommand{\RR}{\mathbb{R}}

\newcommand{\ZZ}{\mathbb{Z}}

\newcommand{\qop}{\triangleleft}
\newcommand{\Aut}{\mathrm{Aut}}
\newcommand{\Inn}{\mathrm{Inn}}
\newcommand{\Dis}{\mathrm{Dis}}
\newcommand{\GAlex}{\mathrm{GAlex}}
\newcommand{\rank}{\mathrm{rank}}
\newcommand{\Grp}{\mathrm{Grp}}

\newcommand{\Schgraph}{\Gamma}
\newcommand{\dSch}{d^{X \curvearrowleft G}}
\newcommand{\SchgraphInn}{\Gamma^{\Inn}}
\newcommand{\dInn}{d^{\Inn}}
\newcommand{\SchgraphDis}{\Gamma^{\Dis}}
\newcommand{\dDis}{d^{\Dis}}

\numberwithin{equation}{section}

\usepackage{color}
\definecolor{darkgreen}{cmyk}{1,0,1,.2}
\definecolor{darkorchid}{rgb}{1.0, 0.5, 0}
\definecolor{persimmon}{rgb}{0.93, 0.35, 0.0}
\definecolor{lightgray}{rgb}{0.7,0.7,0.7}

\newdimen\theight
\def\TeXref#1{%
             \leavevmode\vadjust{\setbox0=\hbox{{\tt
                     \quad\quad  {\small \textrm #1}}}%
             \theight=\ht0
             \advance\theight by \lineskip
             \kern -\theight \vbox to
             \theight{\rightline{\rlap{\box0}}%
             \vss}%
             }}%

\begin{document}

\title{Metrics for quandles}
\author{Kohei Iwamoto, Ryoya Kai, and Yuya Kodama}
\date{}

\address{(K. Iwamoto) 
Graduate School of Science and Engineering, Ritsumeikan University, Nojihigashi
1-1-1, Kusatsu, Shiga, 525-8577, Japan}

\email{ra0061ir@ed.ritsumei.ac.jp}

\address{
	(R. Kai) 
    Center for Educational Research of Science and Mathematics,
	Nara University of Education, 
	Takabatake-Cho, Nara City, 
	Nara, 630-8301, Japan}
\email{kai.ryoya.d8@cc.nara-edu.ac.jp}
\email{sw23889b@st.omu.ac.jp}

\address{(Y. Kodama) 
Graduate School of Science and Engineering,
Kagoshima University, 1-21-35, Korimoto, Kagoshima, Kagoshima 890-0065, Japan
}
\email{yuya@sci.kagoshima-u.ac.jp}

\begin{abstract}
  A quandle is an algebraic system originating in knot theory, which can be regarded as a generalization of the operation of conjugation of groups.
  This structure naturally defines two subgroups of its automorphism group, which are called the inner automorphism group and the displacement group, and they act on the quandle from the right. 
  For a quandle with these groups finitely generated, we investigate the graph structures induced by the actions, and the corresponding metric spaces.
  The graph structures are defined by the notion of the Schreier graph, which is a natural generalization of the Cayley graph for a group.
  In particular, the metric associated with the displacement group of an important class of quandles, namely generalized Alexander quandles, has been studied in detail.
  We show that this metric space is quasi-isometric to the displacement group equipped with a word metric.
  Finally, we provide some examples that are quasi-isometric to typical metric spaces.
\end{abstract}

\maketitle

\section{Introduction}\label{Introduction}

A \emph{quandle} is an algebraic system that is a generalization of conjugation in a group.
The notion of quandles appears in many branches of mathematics; for example, in knot theory as an invariant of knots, and in symmetric space theory as a discretization of symmetric spaces.
The axioms of quandles correspond to the fundamental transformations
for knot diagrams called the Reidemeister moves (see \cite{Kamada-2017-SurfaceKnots4Space} for details),
and also correspond to the properties of point symmetries of symmetric spaces (see \cite{Loos-1969-SymmetricSpacesGeneralTheory} for details).
In these fields, one often focuses on finite quandles because they provide explicit and computable knot invariants, and because they are regarded as a discretization of compact symmetric spaces.
On the other hand, there are many interesting examples of countable quandles.
The knot quandle of a non-trivial knot in the $3$-sphere is countable.
This quandle is a knot invariant defined in a manner similar to the fundamental group.
Additionally, discrete subquandles in non-compact symmetric spaces are generally countable quandles.
In this paper, we focus on such countable quandles.

Similar to group theory, it is more difficult to study infinite quandles than finite ones.
One reason is that it is difficult to understand algebraic structures on infinite sets.
Thus, we approach this difficulty by defining another structure for the infinite set.
Here, we recall techniques from geometric group theory.
Finitely generated groups are naturally equipped with another structure different from the group structure:
let $G$ be a finitely generated group and let $S$ be its finite generating set.
Then this gives rise to a graph structure on $G$ called the Cayley graph.
The graph structure induces a metric on $G$ via the path metric.
The metric depends on the choice of the generating set $S$,
but the quasi-isometry class is determined independently of that choice.
In other words, the quasi-isometry invariant for the metric space can be regarded as an invariant of the group.
Hence, these notions are important for the study of finitely generated groups.
The notion of the Cayley graph is generalized to the Schreier graph, whose set of vertices is a set, and whose edges are defined by a group action on the set
(see \cref{subsec:SchGraph}).
Connected components of this graph 
correspond to the orbits of the action,
and they become metric spaces via the path metric 
induced by the graph structure.
The metrics depend on the choice of the generating set,
but quasi-isometry classes of metric spaces are 
uniquely determined up to that choice
(see \cref{prop:qi_indep_ch_genset_sch_graph}).

A quandle structure naturally defines two groups acting on the quandle.
One of these is called the \emph{inner automorphism group}, which is a group generated by the point symmetries.
The other is called the \emph{displacement group}, which roughly corresponds to the identity component of the inner automorphism group.
In fact, for a connected symmetric space, the displacement group coincides with the identity component of the inner automorphism group.
In this paper, we introduce metrics on a quandle using the Schreier graphs 
with respect to their natural actions.
\begin{dfn}[\cref{def:sch_graph_qdle_inn_act,def:sch_graph_qdle_dis_act}]
  Let $X$ be a quandle.
  \begin{enumerate}
    \item 
    The Schreier graph with respect to the action of the inner automorphism group $\Inn(X)$ with a generating set $A$ is denoted by $\SchgraphInn_A$ and is called the \emph{inner graph}.
    The induced metric on each connected component is denoted by $\dInn_A$, and is called the \emph{inner metric}.

    \item 
    The Schreier graph with respect to the action of the displacement group $\Dis(X)$ with a generating set $U$ is denoted by $\SchgraphDis_U$, and is called the \emph{displacement graph}.
    The induced metric on each connected component is denoted by $\dDis_U$, and is called the \emph{displacement metric}.
  \end{enumerate}
\end{dfn}
In particular, this idea for case (1) is a generalization of the notion of the Cayley graph for a quandle, which was defined by \cite{Winker-1984-QUANDLESKNOTINVARIANTSNFOLD}, and has been studied by several researchers \cite{Hoste-2018-EnumerationProcessRacks,Crans-2019-FiniteNquandlesTorusTwobridge,Mellor-2021-NquandlesLinksa}.
By rephrasing the properties of the Schreier graph in terms of our graph of quandles, we immediately obtain the following.

\begin{thm}[{\cref{thm:qi_indep_ch_genset_sch_graph_inn,thm:qi_indep_ch_genset_sch_graph_dis}}]\label{quandle qi class}
  Let $X$ be a quandle, and let $O \subset X$ be a connected component. 
  Then the following hold:
  \begin{enumerate}[font=\normalfont]
      \item For any finite generating sets $A, B \subset \Inn(X)$, the identity map 
      \begin{align*} 
          \mathrm{id}|_{O}\colon (O, d^{\Inn}_{A}) \to (O, d^{\Inn}_{B}) 
      \end{align*}
      is a quasi-isometry.
      \item For any finite generating sets $U, V \subset \Dis(X)$, the identity map 
      \begin{align*} 
          \mathrm{id}|_{O}\colon (O, d^{\Dis}_{U}) \to (O, d^{\Dis}_{V}) 
      \end{align*}
      is a quasi-isometry.
  \end{enumerate}
\end{thm}

Therefore, we can now investigate the geometry of quandles. 
We note that for a finitely generated quandle, the inner automorphism group is finitely generated, but the displacement group may not be finitely generated.
By \cref{quandle qi class},
if both the inner automorphism group 
and the displacement group are finitely generated,
then we have two quasi-isometry classes for the quandle.
In general, they are not quasi-isometric.

\begin{thm}[{\cref{thm:fgqdle_sch_graph_inn_dis_notqi}}]\label{Inn vs Dis}
    There exist a quandle $X$ and a connected component $O \subset X$ which satisfy the following properties:
    \begin{enumerate}[font=\normalfont]
        \item The groups $\Inn(X)$ and $\Dis(X)$ are finitely generated.
        \item For any finite generating set $A$ of $\Inn(X)$ and $U$ of $\Dis(X)$, the metric spaces $(O, \dInn_A)$ and $(O, \dDis_{U})$ are not quasi-isometric.
    \end{enumerate}
\end{thm}

Next, we consider the special case where the displacement group acts freely on a connected component.
In this situation, the displacement metric on the connected component is identified with the word metric of the displacement group.
We note that the inner automorphism group cannot act freely due to the first axiom of quandles.

\begin{thm}[\cref{quandle-theoretic Milnor-Svarz}]\label{free action}
  Let $X$ be a quandle with a connected component $O$ on which the displacement group $\Dis(X)$ freely acts.
  If the group $\Dis(X)$ is finitely generated, then the metric space $O$ with a displacement  metric is quasi-isometric to the metric space $\Dis(X)$ with a word metric.
\end{thm}

\cref{free action} reminds us of the Milnor--\v{S}varc lemma (for instance, see \cite{Loh-2017-GeometricGroupTheoryIntroduction}).
In fact, both \cref{free action} and the Milnor--\v{S}varc Lemma give a quasi-isometry from a group to a quandle and a metric space, respectively, by using a similar map.
In  geometric group theory, the Milnor--\v{S}varc Lemma provides many examples of quasi-isometries between groups and typical metric spaces.
Analogously, \cref{free action} provides several examples of quandles whose connected components are quasi-isometric to typical metric spaces.

Next, we consider quandles that can apply \cref{free action}, that is, their displacement groups act freely on a connected component of them.
We show that such quandles are essentially \emph{generalized Alexander quandles} (see \cref{quandle with transitive action}).
This type of quandle is a group equipped with a quandle structure given by a group automorphism, which is studied in detail in \cite{Higashitani-2024-GeneralizedAlexanderQuandlesFinite,Higashitani-2024-ClassificationGeneralizedAlexanderQuandles}.
It is important in quandle theory, for example, every homogeneous quandle is presented as a quotient of a generalized Alexander quandle.
In particular, it is known that any group object in the category of quandles is isomorphic to a generalized Alexander quandle.
By \cref{free action}, 
the displacement metric on any connected component of a generalized Alexander quandle is rephrased to the word metric of the displacement group.
\begin{thm}[\cref{thm:qi_GAlex}]\label{intro:qi_GAlex}
  Let $G$ be a group and let $\sigma$ be its group automorphism.
  If the displacement group of the generalized Alexander quandle $X \coloneqq \GAlex(G, \sigma)$ is finitely generated, then any connected component $O$ of the quandle with a displacement metric is quasi-isometric to the displacement group with a word metric. 
\end{thm}
Finally, we give examples of quandles whose connected components with the metrics defined in this paper are quasi-isometric to typical metric spaces, the trees, the Euclidean spaces, the hyperbolic plane, and some $3$-dimensional homogeneous spaces.

This paper is organized as follows:
in \cref{sec:preliminaries}, we review the notion of quandles and quasi-isometries.
In particular, we introduce the Schreier graphs and show that the graphs determine quasi-isometry classes of metric spaces on the set in a general situation.
In \cref{sec:metric_qdle}, we study the Schreier graphs with respect to the natural actions given by the quandle structure.
We show some fundamental properties for such graphs and metrics.
Moreover, we give an example of a quandle to see the difference between quandles with inner and displacement metrics.
In \cref{sec:generalized Alexander quandle}, we focus on the generalized Alexander quandles and determine the quasi-isometry classes for each connected component of these quandles with displacement metrics.
In \cref{sec:examples}, by using the results in \cref{sec:generalized Alexander quandle}, we end this paper with some examples of quandles quasi-isometric to typical metric spaces:
trees, the Euclidean spaces, the hyperbolic plane, and the $3$-dimensional homogeneous spaces.

\section{Preliminaries}\label{sec:preliminaries}
\subsection{Quandles}\label{subsec:qdle}
In this subsection, we review some notions of quandles and properties related to group actions on quandles.
The following definitions were originally given by Joyce \cite{Joyce-1982-ClassifyingInvariantKnotsKnot}.
First, let us recall the definition of quandles.
\begin{dfn}\label{def:qdle}
    A non-empty set $X$ equipped with a binary operation $\qop$ is called a \emph{quandle} if the following conditions hold:
    \begin{enumerate}
        \item $x \qop x = x$ holds for any $x \in X$.
        \item The map $s_y \colon X \to X$ defined by $s_y(x) \coloneqq x \qop y$ is a bijection for any $y \in X$.
        \item $(x \qop y) \qop z = (x \qop z) \qop (y \qop z)$ holds for any $x, y, z \in X$.
    \end{enumerate}
    The bijection $s_y \colon X \rightarrow X$ is called the \emph{point symmetry} at $y$. 
    We denote $x \qop^{-1} y \coloneqq s_y^{-1}(x)$.
\end{dfn}

A subset of $X$ is called a \emph{subquandle} if it is closed under both the binary operation $\qop$ and its inverse $\qop^{-1}$.
Let $(X, \qop)$ and $(Y, \qop^{\prime})$ be quandles. 
A map $f \colon X \to Y$ is called a \emph{quandle homomorphism} if it satisfies $f(x \qop x^{\prime}) = f(x) \qop^{\prime} f(x^{\prime})$ for any elements $x, x^{\prime} \in X$. 
A bijective quandle homomorphism $f \colon X \to Y$ is called a \emph{quandle isomorphism}.
Two quandles $X$ and $Y$ are said to be \emph{isomorphic} if there exists a quandle isomorphism from $X$ to $Y$.
The set of all quandle isomorphisms from $X$ to itself is denoted by $\Aut(X)$.
This set with the binary operation $fg \coloneqq g \circ f$ forms a group, and is called the \emph{automorphism group}.
The group acts on the quandle $X$ from the right by $x \cdot f \coloneqq f(x)$.
A quandle $X$ is said to be \emph{homogeneous} if $\Aut(X)$ acts transitively on $X$.
Note that any point symmetry $s_y \colon X \to X$ is a quandle isomorphism by the second and third axioms of quandles.
The set of point symmetries generates a subgroup of $\Aut(X)$, which plays an important role in this paper. 
In addition, the connectedness of a quandle is defined by the action of this group.
\begin{dfn}\label{def:qdle_aut_grp_hmg_qdle}
    The \emph{inner automorphism group} of $X$ is a subgroup in $\Aut(X)$ generated by $\{ s_y \mid y \in X \}$, and is denoted by $\Inn(X)$.
    A \emph{connected component} of $X$ is an orbit under the action of $\Inn(X)$.
    The set of all connected components of $X$ is denoted by $\pi_{0}(X) = X / \Inn(X)$.
    A quandle $X$ is said to be \emph{connected} if $\Inn(X)$ acts transitively on $X$.
\end{dfn}

The following group was defined by Joyce \cite{Joyce-1982-ClassifyingInvariantKnotsKnot} as the transvection group. 
This group also plays a central role in this paper.
\begin{dfn}\label{def:displacement_grp}
    The \emph{displacement group} of a quandle $X$ is 
    the subgroup of $\Aut(X)$ generated by the set $\{s_x s_y^{-1} \mid x, y \in X\}$,
    and is denoted by $\Dis(X)$.
\end{dfn}
The following are some properties of these actions of groups.
\begin{prop}[{\cite[Proposition 2.1]{Hulpke-2016-ConnectedQuandlesTransitiveGroups}}]\label{prop:properties_displ_grp}
    Let $X$ be a quandle.
    Then the groups $\Inn(X)$ and $\Dis(X)$ satisfy the following properties:
    \begin{enumerate}[font=\normalfont]
        \item The groups $\Inn(X)$ and $\Dis(X)$ are normal subgroups of $\Aut(X)$.
        \item The group $\Dis(X)$ is a normal subgroup of the group $\Inn(X)$, and the quotient group $\Inn(X) / \Dis(X)$ is cyclic. 
        \item The following equality holds:
        \begin{align*}
            \Dis(X) = \left\{ s_{a_1}^{k_1} \cdots s_{a_n}^{k_n} \mid n \geq 0,\; a_i \in X,\; \sum_{i=1}^n k_i = 0 \right\}.
        \end{align*}
        \item The actions of $\Inn(X)$ and $\Dis(X)$ on $X$ have the same orbits.
    \end{enumerate}
\end{prop}
By \cref{prop:properties_displ_grp}, each connected component of a quandle is equal to an orbit under the action of $\Dis(X)$.

We end this subsection with some examples of quandles. 
The first example is one of the simplest infinite quandles. 
This quandle gives a difference of metrics introduced in \cref{sec:metric_qdle}.
\begin{ex}\label{ex:infty_dihedral_qdle}
Let $\ZZ$ be the set of all integers.
We define a binary operation $\qop$ on $\ZZ$ by 
\begin{align*}
    x \qop y \coloneqq 2y - x \qquad \text{for all } x, y \in \mathbb{Z}.
\end{align*}
It defines a quandle structure on $\ZZ$.
This quandle is called the \emph{infinite dihedral quandle}, and is denoted by $R_\infty$.
Moreover, the following hold:
\begin{enumerate}
    \item The point symmetries at the points $0, 1 \in R_{\infty}$ satisfy that
    \begin{align*}
        x \cdot s_{0} = -x, \qquad x \cdot s_{1} = 2 - x, \qquad x \cdot (s_{1}s_{0}^{-1}) = x - 2.
    \end{align*}
    \item The set $A = \{s_{0}, s_{1}\}$ generates the group $\Inn(R_{\infty})$.
    \item The set $U = \{s_{1}s_{0}^{-1}\}$ generates the group $\Dis(R_{\infty})$.
    \item $\pi_{0}(R_{\infty}) = \{ O_{\mathrm{even}},\, O_{\mathrm{odd}}\}$, 
    where $O_{\mathrm{even}} \coloneqq 2\mathbb{Z}, \quad O_{\mathrm{odd}} \coloneqq 2\mathbb{Z} + 1$.
    \item The quandle $R_{\infty}$ is homogeneous.
\end{enumerate}
\end{ex}

\begin{proof}
    One can easily check (1) by a direct calculation. 
    Here, any $z \in R_{\infty}$ satisfies that
    \begin{align*}
        s_z = (s_1s_0^{-1})^z s_0. 
    \end{align*}
    This shows that the set $A$ is a generating set of $\Inn(X)$, thus we obtain (2).
    Moreover, a generator $s_{x}s_{y}^{-1}$ of the displacement group for $x, y \in R_{\infty}$ is given by
    \begin{align*}
         z \cdot (s_x s_y^{-1}) = (2x - z) \cdot s_{y}^{-1} = 2y - (2x - z) = z - 2(x-y) = z \cdot (s_{1}s_{0}^{-1})^{x - y},
    \end{align*}
    for any $z \in R_{\infty}$.
    Therefore the set $U$ generates the group $\Dis(R_{\infty})$, and we have (3).
    We now show that the connected component including $0$ is equal to $O_{\mathrm{even}}$, that is, the equation $0 \cdot \Inn(R_{\infty}) = O_{\mathrm{even}}$ holds.
    Since generators $s_0$ and $s_1$ of $\Inn(R_{\infty})$ preserve the set $O_{\mathrm{even}}$, and the element $0$ is in $O_{\mathrm{even}}$, we have $0 \cdot \Inn(R_{\infty}) \subset O_{\mathrm{even}}$. 
    Conversely, any element $2n \in O_{\mathrm{even}}$ is in the orbit of $0$ as $2n = 0 \cdot (s_{1}s_{0}^{-1})^{-n}$. 
    Hence, the assertion holds.
    By the same argument, one can show that the connected component including $1$ is equal to $O_{\mathrm{odd}}$.
    This shows (4).
    For an integer $n \in R_{\infty}$, the map $\alpha \colon R_{\infty} \to R_{\infty}$ defined by $\alpha(z) = z + n$ is a quandle isomorphism and satisfies $\alpha(0) = n$.
    Hence, the quandle $R_{\infty}$ is homogeneous, which completes the proof.
\end{proof}

Finally, we give a general class of quandles that gives many examples from groups.
The binary operation of quandles can be regarded as a generalization of the conjugation of groups.
In fact, the operation defines a quandle structure on a subset of a group.
\begin{ex}\label{eq:conj_qdle}
    Let $G$ be a group, and let $X \subseteq G$ be a nonempty subset that is closed under conjugation. 
    More precisely, the element $g^{-1} x g$ is in $X$ for any $x \in X$ and $g \in G$.
    Then, the set $X$ is a quandle equipped with a binary operation $\qop$ defined by
    \begin{align*} 
        x \qop y \coloneqq y^{-1} x y, \qquad x, y \in X, 
    \end{align*} 
    and is called the \emph{conjugation quandle}.
\end{ex}

\subsection{Schreier graphs}\label{subsec:SchGraph}

In this subsection, we introduce the notion of Schreier graphs.
This graph is regarded as a generalization of the Cayley graph and will be used to define some metrics for quandles. 
First, we give the definition of the Schreier graph, and then we show that the quasi-isometry class of the metric induced from the graph structure on each component is uniquely determined up to the choice of the finite generating sets of the group acting on it.
\begin{dfn}\label{def:schreier_graph} 
    Let $X$ be a nonempty set equipped with a right action of a group $G$, and let $S \subset G$ be a generating set. 
    The \emph{Schreier graph} of the right action of $G$ with respect to $S$ is the undirected graph $\Schgraph(X, G, S)$ 
    which is defined as follows:
    \begin{enumerate}[font=\normalfont]
        \item The set of vertices is $X$.
        \item Two vertices $x$ and $y$ are connected by an edge if it satisfies that $y = x \cdot s$ for some $s \in S \cup S^{-1}$, where $S^{-1}\coloneqq \{s^{-1} \mid s \in S\}$. 
        Then the edge is labeled by $s$. 
    \end{enumerate}
\end{dfn}

\begin{rem}
The Cayley graph of a finitely generated group $G$ can be regarded as a special case of the Schreier graph. 
In fact, the Schreier graph of the natural right action of $G$ on the set $G$ defined by $g \cdot h = gh$ for $g, h \in G$ is the Cayley graph with a certain generating set.
\end{rem}

The connected components of the Schreier graph coincide precisely with the orbits of the group action, as shown below.
One may easily show this, 
but we give a proof for the readers.
\begin{prop}\label{prop:conn_cmpts_sch_graph}
    Let $X$ be a nonempty set equipped with a right action of a group $G$ with a generating set $S$. 
    Then there exists a bijection from the set $\mathcal{O}$ of connected components of the Schreier graph $\Schgraph(X, G, S)$ to the set $X / G$ of $G$-orbits in $X$.
\end{prop}
\begin{proof}
To begin the proof, we define the map
\begin{align*} 
    \Phi \colon \mathcal{O} \to X/G,\quad \mathcal{O}_x \mapsto O_x, 
\end{align*}
where $\mathcal{O}_x$ is the connected component of $\Schgraph(X, G, S)$ containing $x$, and $O_x = x \cdot G$ is the $G$-orbit of $x$.
Firstly, we show that $\Phi$ is well-defined.
Suppose that $y \in \mathcal{O}_x$.
By the definition of connected components in the Schreier graph, there exists a finite path connecting $x$ and $y$
\begin{align*}
x = x_0,x_1, \cdots, x_n = y,
\end{align*}
where each edge corresponds to a generator $s_k \in S$ and a sign $\varepsilon_k \in \{\pm 1\}$ such that
\begin{align*}
    x_{k+1} = x_k \cdot s_k^{\varepsilon_k}, \quad 0 \leq k \leq n-1.
\end{align*}
This path corresponds to the group element $g \in G$ which is described as
\begin{align*}
    g = s_0^{\varepsilon_0} s_1^{\varepsilon_1} \cdots s_{n-1}^{\varepsilon_{n-1}}, \quad s_k \in S, \quad \varepsilon_k \in \{\pm 1\},
\end{align*}
and we have $y = x \cdot g$.
Thus, $x$ and $y$ belong to the same orbit.
Therefore, $\Phi$ is well-defined.

Secondly, we show that $\Phi$ is injective.
We assume that $\Phi(\mathcal{O}_x) = \Phi(\mathcal{O}_y)$.
So we have $O_x = O_y$ by the definition of $\Phi$ and, there exists an element $g \in G$ such that $y = x \cdot g$.
Since $S$ is a generating set of $G$, the group element $g$ can be written as
\begin{align*}
    g = s_0^{\varepsilon_0}s_1^{\varepsilon_1} \cdots s_{n-1}^{\varepsilon_{n-1}}, \quad \text{with } s_i \in S,\ \varepsilon_i \in \{\pm 1\}.
\end{align*}
We define a sequence of vertices in $X$ by
\begin{align*}
    x_0 \coloneqq x, \quad x_{k+1} = x_k \cdot s_k^{\varepsilon_k}, \quad 0 \leq k \leq n-1.
\end{align*}
Then, by construction, we have $x_n = y$, and each pair $(x_{k}, x_{k+1})$ is connected by an edge in the Schreier graph $\Schgraph(X, G, S)$.
Therefore, there exists a path from $x$ to $y$ in the Schreier graph, so $x$ and $y$ belong to the same connected component.
That is, $\mathcal{O}_x = \mathcal{O}_y$.
Hence, $\Phi(\mathcal{O}_x) = \Phi(\mathcal{O}_y)$ implies $\mathcal{O}_x = \mathcal{O}_y$,
which shows that $\Phi$ is injective.

Finally, we show that $\Phi$ is surjective.
Let $O_x$ be any $G$-orbit in $X$.
Any point $y \in O_x$ can be written as $y = x \cdot g$ for some $g \in G$, and since $S$ generates $G$, we have
\begin{align*}
    g = s_0^{\varepsilon_0}s_1^{\varepsilon_1} \cdots s_{n-1}^{\varepsilon_{n-1}}, \quad \text{with } s_i \in S,\ \varepsilon_i \in \{\pm 1\}.
\end{align*}
It follows that $x$ and $y$ lie in the same connected component of the Schreier graph $\Schgraph(X, G, S)$, so $y \in \mathcal{O}_x$.
Hence, $\Phi(\mathcal{O}_x) = O_x$, and $\Phi$ is surjective.
\end{proof}

The graph structure of $\Schgraph(X, G, S)$ provides a metric on each connected component, in other words, each $G$-orbit, as the path metric.
More precisely, for two vertices $u$ and $v$ in the same component $O$, we define 
\begin{align*}
    \dSch_{S}(u, v) = \inf\{l(\gamma) \mid \gamma \colon \text{a path on $O$ connecting $u$ and $v$}\}
\end{align*}
where we assign the length of each edge as one, and denote the length of a path $\gamma$ by $l(\gamma)$.
This metric depends on the choice of the generating set $S$.
To avoid this problem, we recall the definition of the quasi-isometry.
\begin{dfn}\label{dfn:qi_metric_sp}
    Let $(X_{1}, d_{X_1})$ and $(X_{2}, d_{X_{2}})$ be metric spaces. 
    \begin{enumerate}
        \item A map $f \colon X_{1} \rightarrow X_{2}$ is called a \emph{quasi-isometric embedding} if there exist constants $\lambda \geq 1$ and $k \geq 0$ such that any $x, y \in X_{1}$ satisfy that
        \begin{equation*}
            \frac{1}{\lambda} d_{X_{1}}(x, y)-k \leq d_{X_{2}}(f(x), f(y)) \leq \lambda d_{X_{1}}(x, y) + k.
        \end{equation*}

        \item A map $f \colon X_{1} \rightarrow X_{2}$ has \emph{coarsely dense image} if there exists a constant $C \geq 0$ such that for every $y \in X_{2}$, there exists some $x \in X_{1}$ satisfying  $d_{X_{2}}(f(x), y) \leq C$.
        
        \item A map $f \colon X_{1} \rightarrow X_{2}$ is called a \emph{quasi-isometry} if it is a quasi-isometric embedding and has coarsely dense image. 
        
        \item Two metric spaces $(X_{1}, d_{X_{1}})$ and $(X_{2}, d_{X_{2}})$ are said to be \emph{quasi-isometric} if there exists a quasi-isometry between them.
    \end{enumerate}
\end{dfn}

If $\lambda=1$ and $k=0$, this is an isometry. 
The relation of being quasi-isometric defines an equivalence relation among metric spaces.
We show that the quasi-isometry class of $\dSch_{S}$ is uniquely determined up to the choice of the generating sets $S$.
We believe this is a well-known fact, but give a proof for the convenience of the reader.
\begin{lem}\label{prop:qi_indep_ch_genset_sch_graph}
    Let $X$ be a nonempty set equipped with a right action of a finitely generated group $G$.
    If finite subsets $S$ and $T$ generate the group $G$,  
    then for any $G$-orbit $O$ in $X$, the identity map $\mathrm{id} \colon (O, \dSch_{S}) \to (O, \dSch_{T})$ is a quasi-isometry.
\end{lem}
\begin{proof}
By \cref{prop:conn_cmpts_sch_graph}, the $G$-orbit coincides with the vertex set of the connected component of the Schreier graph $\Schgraph(X, G, S)$. 
Let $|\cdot|_S$ and $|\cdot|_T$ denote the word lengths on $G$ with respect to $S$ and $T$, respectively. 
We define:
\begin{align*}
    L_{S,T} \coloneqq \max\left( \{ |s|_T \mid s \in S \cup S^{-1} \} \cup \{1\} \right), \quad
    L_{T,S} \coloneqq \max\left( \{ |t|_S \mid t \in T \cup T^{-1} \} \cup \{1\} \right),
\end{align*}
which are finite real numbers and satisfy $L_{S,T} \ge 1$, $L_{T,S} \ge 1$.
Let $x, y \in O$, and assume $\dSch_S(x, y) = n$. 
Then there exist elements $s_1, \dots, s_n \in S \cup S^{-1}$ such that $y = x \cdot (s_1 s_2 \cdots s_n)$,
and $g \coloneqq s_1 s_2 \cdots s_n \in G$ satisfies $x \cdot g = y$.
Since each $s_i \in S \cup S^{-1}$ satisfies $|s_i|_T \le L_{S,T}$, we obtain
\begin{align*}
    |g|_T \le |s_1|_T + \cdots + |s_n|_T \le n  L_{S,T}.
\end{align*}
It follows that
\begin{align*}
    \dSch_T(x, y) \le |g|_T \le L_{S,T}  \dSch_S(x, y).
\end{align*}
By a similar argument, we also have
\begin{align*}
    \dSch_S(x, y) \le L_{T,S} \dSch_T(x, y).
\end{align*}
Let $L \coloneqq \max\{L_{S,T}, L_{T,S}\}$. 
Then for all $x, y \in O$, we have
\begin{align*}
    \frac{1}{L} \dSch_S(x, y) \le \dSch_T(x, y) \le L  \dSch_S(x, y),
\end{align*}
which shows that the identity map
\begin{align*}
    \mathrm{id}|_{O} \colon (O, \dSch_S) \longrightarrow (O, \dSch_T)
\end{align*}
is a quasi-isometry.
\end{proof}

Finally, we introduce a quasi-isometry invariant, the number of ends.
This will be used to show that two metric spaces are not quasi-isometric in \cref{thm:fgqdle_sch_graph_inn_dis_notqi}.
Let $\Gamma$ be a connected, locally finite graph, and fix a basepoint $v_0 \in \Gamma$. 
We denote the ball of radius $n$ centered at $v_0$ by $B(n)$,
and the number of unbounded connected components of the set $C$ by $\|C\|$. 
\begin{dfn}\label{def:number_ends_graph}
\emph{The number of ends} of $\Gamma$ is defined as
\begin{align*}
    e(\Gamma) \coloneqq \lim_{n \to \infty} \|\Gamma \setminus B(n)\|.
\end{align*}
\end{dfn}
The sequence $\|\Gamma \setminus B(n)\|$ has a limit in $[0, \infty]$.
It is easy to see that this is independent of the choice of the basepoint.
In fact, this is a quasi-isometry invariant \cite{Loh-2017-GeometricGroupTheoryIntroduction}.
\section{The Schreier graphs of quandles}\label{sec:metric_qdle}
In this section, we define graph structures for a quandle $X$ by using the framework of the Schreier graphs. 
We especially focus on the actions of the inner automorphism group $\Inn(X)$ and the displacement group $\Dis(X)$ since each connected component of the Schreier graphs of these actions corresponds to each connected component of a quandle (see \cref{prop:conn_cmpts_sch_graph_inn,prop:conn_cmpts_sch_graph_dis}).
The Schreier graph associated with the action of $\Inn(X)$ is a generalization of the diagram of quandles, which is also called the Cayley graph of quandles, introduced by Winker \cite{Winker-1984-QUANDLESKNOTINVARIANTSNFOLD}.
These graphs induce metrics on a quandle.
In particular, the quasi-isometry classes of the metric spaces are uniquely determined up to the choice of the generating sets if the corresponding group is finitely generated.
In particular, if the displacement group acts freely, then the metric associated with the action can be translated into its word metric.
If both of the groups $\Inn(X)$ and $\Dis(X)$ are finitely generated, then we have two quasi-isometry classes of metric spaces.
However, these metric spaces are not quasi-isometric in general.

\subsection{The case of inner automorphism groups}\label{subsec:SchInn}
In this subsection, we focus on the action of the inner automorphism group on a quandle.
We define a graph structure for a quandle as the associated Schreier graph and then investigate several properties of this structure.
\begin{dfn}\label{def:sch_graph_qdle_inn_act}
  Let $X$ be a quandle, and let $A$ be a generating set of the group $\Inn(X)$.
  The Schreier graph $\Schgraph(X, \Inn(X), A)$ (see \cref{def:schreier_graph}) is called the \emph{inner graph of $X$ with respect to the generating set $A$}, and is denoted by $\SchgraphInn_A(X)$.
  The path metric on each connected component induced by $\SchgraphInn_A(X)$
  is called the \emph{inner metric}, and is denoted by $\dInn_A$.
\end{dfn}
This graph can be regarded as a generalization of the diagram of quandle defined by Winker \cite[Definition 4.3.9]{Winker-1984-QUANDLESKNOTINVARIANTSNFOLD}. 
The diagram is determined by the operation of a quandle from the right and a generating set of the quandle.
In fact, a generating set of a quandle $X$ gives a generating set of the group $\Inn(X)$
as \cref{prop:inn_fingen_qdle_fingen}. 
We recall that a subset $U \subset X$ is a \emph{generating set} of $X$ 
if any $x \in X$ can be presented as $x = (\cdots (u \qop^{\varepsilon_1} u_1) \qop^{\varepsilon_2} \cdots) \qop^{\varepsilon_n} u_n$ for some $u, u_1, \dots u_n \in U$ 
and $\varepsilon_1, \dots, \varepsilon_n \in \{\pm 1\}$.
In particular, a finite generating set of a quandle 
gives a finite generating set of the inner automorphism group.
\begin{prop}\label{prop:inn_fingen_qdle_fingen}
Let $X$ be a quandle with a generating set $A$. 
Then the following hold:
\begin{enumerate}[font=\normalfont]
    \item The set $s(A) = \{ s_a \mid a \in A \}$ generates the inner automorphism group $\Inn(X)$.
    \item If $X$ is a finitely generated quandle, then $\Inn(X)$ is a finitely generated group.
\end{enumerate}
\end{prop}
\begin{proof}
First, we show (1).
It is enough to show that any point symmetry $s_x$ at $x \in X$ can be expressed as the product of finite elements in $s(A) \cup (s(A))^{-1}$. 
Since $A$ is a generating set, any $x \in X$ is presented as $x = (\cdots (a \qop^{\varepsilon_1} a_1) \qop^{\varepsilon_2} \cdots) \qop^{\varepsilon_n} a_n$ for some $a, a_1, \dots a_n \in A$ 
and $\varepsilon_1, \dots, \varepsilon_n \in \{\pm 1\}$, 
From the axioms of quandles, the point symmetries satisfy that
\begin{align*}
    s_{u \triangleleft v} = s_v^{-1} s_u s_v, \qquad s_{u \triangleleft^{-1} v} = s_v s_u s_v^{-1} \qquad (u, v \in X).
\end{align*}
By repeatedly applying these relations, the point symmetry $s_x$ can be expressed by
\begin{align*}
    s_x = s_{a_n}^{-\varepsilon_n} \cdots s_{a_1}^{-\varepsilon_1} s_a\, s_{a_1}^{\varepsilon_1} \cdots s_{a_n}^{\varepsilon_n}.
\end{align*}
Thus, we have (1).

If the quandle $X$ is finitely generated, then there exists a finite generating set $A = \{a_1, \dots, a_n\} \subset X$. 
Then, the set $s(A)$ is finite and generates the group $\Inn(X)$ by (1), which completes the proof.
\end{proof}
In this context, the connectedness of the graphs coincides with the connectedness of quandles.
\begin{lem}\label{prop:conn_cmpts_sch_graph_inn}
    Let $X$ be a quandle and let $A$ be a generating set of the group $\Inn(X)$.
    Then the set of vertices in a connected component of the inner graph $\SchgraphInn_{A}(X)$ is a connected component of the quandle $X$, and the converse also holds.
\end{lem}
\begin{proof}
    Each connected component of a graph $\SchgraphInn_{A}(X)$ is an $\Inn(X)$-orbit by \cref{prop:conn_cmpts_sch_graph}.
    By definition, an $\Inn(X)$-orbit is a connected component of a quandle $X$, as desired.
\end{proof}
Note that even if the inner automorphism group $\Inn(X)$ is finitely generated, the quandle $X$ itself may not be so, for example, the infinite trivial quandle.

Here, we restate the result from \cref{subsec:SchGraph} in the case of the action of $\Inn(X)$.
\begin{thm}\label{thm:qi_indep_ch_genset_sch_graph_inn}
    Let $X$ be a quandle, and let $O$ be its connected component. 
    If subsets $A$ and $B$ are finite generating sets of $\Inn(X)$, then the identity map 
    \begin{align*} 
        \mathrm{id}|_{O}\colon (O, \dInn_{A}) \to (O, \dInn_{B}) 
    \end{align*}
    is a quasi-isometry.
\end{thm}
\begin{proof}
    It immediately follows from \cref{prop:qi_indep_ch_genset_sch_graph}.
\end{proof}
The above theorem allows us to introduce a well-defined notion of quasi-isometry between quandle components with respect to the action of $\Inn(X)$.
To discuss quasi-isometry of quandles with respect to the inner automorphism group, we introduce the following conventions.
Let $X$ and $Y$ be quandles whose inner automorphism groups $\Inn(X)$ and $\Inn(Y)$ are finitely generated.
A connected component $O$ of $X$ is said to be \emph{quasi-isometric with respect to an inner metric} to a metric space $M$ if the metric space $(O, \dInn_A)$ is quasi-isometric to $M$ for some finite generating set $A$ of $\Inn(X)$.
Two quandles $X$ and $Y$ are \emph{quasi-isometric with respect to an inner metric} if there exists a map $f \colon X \to Y$ such that the induced map between the sets of connected components $\Bar{f} \colon \pi_{0}(X) \to \pi_{0}(Y)$ is a bijection and the restriction map $f|_{O} \colon O \to \Bar{f}(O)$ to any connected component $O$ of $X$ is a quasi-isometry with respect to the inner metric.

The following result states that any two connected components of a homogeneous quandle are quasi-isometric to each other.
\begin{cor}\label{cor:all_conn_cmpts_homqdle_qi_inn}
    Let $X$ be a quandle.
    Suppose that the inner automorphism group $\Inn(X)$ is finitely generated and $A \subset \Inn(X)$ is a finite generating set.
    If the quandle $X$ is homogeneous, then any two connected components $O$ and $O^{\prime}$ of $X$ with the inner metrics $\dInn_{A}$ are quasi-isometric.
\end{cor}
\begin{proof}
First, we show that there exists an automorphism $f \in \Aut(X)$ such that $f(O) = O^{\prime}$.
We take base points $x \in O$ and $x^{\prime} \in O^{\prime}$.
Since the quandle $X$ is homogeneous, there exists an automorphism $f \in \Aut(X)$ such that $f(x) = x^{\prime}$. 
Let $y \in O$.
Then, by the definition of connected components for quandles, there exist elements $\alpha_1, \dots, \alpha_n \in \Inn(X)$ such that $y = x \cdot \alpha_1 \cdots \alpha_n$. 
Here, we note that any point symmetry $s_x$ at $x \in X$ satisfies $f^{-1} s_x f = s_{f(x)}$ because the map $f$ is a quandle isomorphism.
Thus, we have 
\begin{align*}
    f(y) = f(x \cdot \alpha_1 \cdots \alpha_n) = f(x) \cdot f^{-1} \alpha_1 f \cdots f^{-1} \alpha_n f = x^{\prime} \cdot f^{-1} \alpha_1 f \cdots f^{-1} \alpha_n f.
\end{align*}
Since each conjugate $f^{-1} \alpha_i f$ belongs to $\Inn(X)$,
we have $f(y) \in O^{\prime}$.
Hence, we obtain $f(O) \subset O^{\prime}$. 
Applying the same argument to $f^{-1} \in \Aut(X)$, we obtain $f^{-1}(O^{\prime}) \subset O$, which implies $O^{\prime} \subset f(O)$.
Therefore $f(O) = O^{\prime}$.

We define a set by $B \coloneqq f^{-1} A f$.
Since $f$ is a quandle isomorphism, each conjugate $f^{-1} a f$ with $a \in A$ belongs to $\Inn(X)$. 
Therefore, we have $B \subset \Inn(X)$.
We now claim that $B$ generates $\Inn(X)$.
Let $g \in \Inn(X)$.
Since $A$ generates $\Inn(X)$ there exist $a_1, \cdots, a_n \in A$ and $\varepsilon_1, \cdots, \varepsilon_n \in \{\pm1\}$ such that $f g f^{-1} = a_1^{\varepsilon_1} \cdots a_n^{\varepsilon_n}$.
If we put $b_i \coloneqq f^{-1} a_i f \in B$, then we have
\begin{align*}
    g = f^{-1} a_1^{\varepsilon_1} \cdots a_n^{\varepsilon_n} f = (f^{-1} a_1 f)^{\varepsilon_1} \cdots (f^{-1} a_n f)^{\varepsilon_n} = b_1^{\varepsilon_1} \cdots b_n^{\varepsilon_n}.
\end{align*}
Therefore, $B$ is a generating set of $\Inn(X)$.

We now show that the restriction map $f \colon (O, \dInn_{A}) \to (O^{\prime}, \dInn_{B})$ is an isometry.
Let us take $x, y \in O$ with the distance $\dInn_{A}(x, y)= m$. 
Thus, there exists a path through $m$ edges from $x$ to $y$.
We denote the labels of edges by $a_i \in A$. 
Then by tracing the labels along the path, we have $y = x \cdot (a_1^{\varepsilon_{1}} \cdots a_{m}^{\varepsilon_{m}})$.
We denote the $b_i \coloneqq f^{-1} a_i f \in B$, and then obtain 
\begin{align*}
    f(y) &= y \cdot f = (x \cdot (a_1^{\varepsilon_{1}} \cdots a_{m}^{\varepsilon_{m}})) \cdot f
    = (x \cdot f) \cdot(f^{-1} (a_1^{\varepsilon_{1}} \cdots a_{m}^{\varepsilon_{m}}) f)\\
    &= f(x) \cdot \left((f^{-1} a_1^{\varepsilon_{1}}f)\cdots (f^{-1}a_{m}^{\varepsilon_{m}}f)\right)
    = f(x) \cdot \left( b_1^{\varepsilon_{1}} \cdots b_m^{\varepsilon_{m}} \right).
\end{align*}
This shows that $\dInn_B(f(x), f(y)) \leq m = \dInn_A(x, y)$.
By the same argument replacing $f$ and $A$ with $f^{-1}$ and $B$ respectively, we can see that $\dInn_A(x, y) \leq \dInn_B(f(x), f(y)) $.
Thus, we obtain $f \colon (O, \dInn_{A}) \to (O^{\prime}, \dInn_{B})$ is an isometry.

Finally, by the Theorem $\ref{thm:qi_indep_ch_genset_sch_graph_inn}$, there exists a quasi-isometry map $\iota \colon (O^{\prime}, \dInn_{B}) \to (O^{\prime}, \dInn_{A})$.
Since the composition of an isometry and a quasi-isometry is also a quasi-isometry, we conclude that the map $\iota \circ f\colon (O, \dInn_{A}) \to (O^{\prime}, \dInn_{A})$ is a quasi-isometry.
\end{proof}

\subsection{The case of displacement groups}\label{subsec:SchDis}
In this subsection, we focus on the action of the displacement group on a quandle.
We define a graph structure for a quandle via the associated Schreier graph and then investigate several properties of this structure.
In particular, we show that if the displacement group acts freely on a connected component, then it is quasi-isometric to that component.
\begin{dfn}\label{def:sch_graph_qdle_dis_act}
Let $X$ be a quandle,
and let $U$ be a generating set of the group $\Dis(X)$.
The Schreier graph $\Schgraph(X, \Dis(X), U)$ (see \cref{def:schreier_graph})
is called the \emph{displacement graph of $X$ with respect to the generating set $U$}, and is denoted by $\SchgraphDis_U(X)$.
The path metric on each connected component induced by $\SchgraphDis_U(X)$ is called the \emph{displacement metric}, and is denoted by $\dDis_U$.
\end{dfn}
Similar to \cref{prop:conn_cmpts_sch_graph_inn}, the following result also holds in the case of the action by the displacement group.
\begin{prop}\label{prop:conn_cmpts_sch_graph_dis}
    Let $X$ be a quandle and let $U$ be a generating set of the group $\Dis(X)$.
    Then the set of vertices in a connected component of the displacement graph $\SchgraphDis_{U}(X)$ is a connected component of the quandle $X$, and the converse also holds.
\end{prop}
\begin{proof}
    By combining \cref{prop:conn_cmpts_sch_graph}, \cref{prop:properties_displ_grp}(4) and \cref{prop:conn_cmpts_sch_graph_inn}, we have the desired result.
\end{proof}
We begin by restating the result from \cref{subsec:SchInn} in the context of the displacement group. As in the case of the inner automorphism group, the quasi-isometry class of each connected component is independent of the choice of the finite generating set for the displacement group.
\begin{thm}\label{thm:qi_indep_ch_genset_sch_graph_dis}
    Let $X$ be a quandle, and let $O$ be its connected component. 
    If subsets $U$ and $V$ are finite generating sets of $\Dis(X)$, then the identity map 
    \begin{align*} 
        \mathrm{id}|_{O}: (O, \dDis_{U}) \to (O, \dDis_{V}) 
    \end{align*}
    is a quasi-isometry.
\end{thm}
\begin{proof}
    It follows immediately from \cref{prop:qi_indep_ch_genset_sch_graph}.
\end{proof}

\begin{rem}
In order to study the quasi-isometry classes defined by the displacement group, it is essential to assume that the displacement group itself is finitely generated.
While \cref{prop:inn_fingen_qdle_fingen} shows that if a quandle is finitely generated, then so is its inner automorphism group, the same does not hold for the displacement group in general.
Indeed, there exist finitely generated quandles whose displacement groups are not finitely generated.
\end{rem}
\begin{proof}
  In this proof, 
  we use some basic facts for a knot $K$ in the $3$-dimensional sphere
  and its knot quandles $Q(K)$.
  See \cite{Kamada-2017-SurfaceKnots4Space,Kawauchi-1996-SurveyKnotTheory} for details.
  Let us assume that $K$ is non-fibered.
  We now show that the displacement group of $Q(K)$ is not finitely generated.
  Note that the quandle $Q(K)$ is finitely generated.
  First, we recall that for any quandle $Q$, there exists a surjective group homomorphism from the associated group $G_Q$ of $Q$ to the inner automorphism group and its kernel is included in the center of $G_Q$.
  In our case, the group $G_{Q(K)}$ is isomorphic to the knot group $G(K) \coloneqq \pi_1(S^3 \setminus K)$.
  Here, the center of the knot group is trivial
  since $K$ is not a torus knot.
  Hence, the inner automorphism group $\Inn(Q(K))$ is isomorphic to $G(K)$.
  Moreover, 
  the displacement group $\Dis(Q(K))$ is isomorphic to 
  the commutator subgroup of $\Inn(Q(K))$ 
  \cite[Proposition 2.3]{Hulpke-2016-ConnectedQuandlesTransitiveGroups}.
  In conclusion, the displacement group $\Dis(Q(K))$
  is isomorphic to the commutator subgroup $[G(K), G(K)]$ of 
  the knot group $G(K)$.
  Since the group $[G(K), G(K)]$ is finitely generated if and only if the knot $K$ is fibered, 
  the knot quandle of a non-fibered knot, for instance, the knot $5_2$, is an example.
\end{proof}

The following conventions and terminology will be used to discuss quasi-isometry of quandles under the action of the displacement group.
Let $X$ and $Y$ be quandles whose displacement groups $\Dis(X)$ and $\Dis(Y)$ are finitely generated.
A connected component $O$ of $X$ is \emph{quasi-isometric with respect to a displacement metric} to a metric space $M$ if the metric space $(O, \dDis_U)$ is quasi-isometric to $M$ for some finite generating set $U$ of $\Dis(X)$.
Two quandles $X$ and $Y$ are \emph{quasi-isometric with respect to a displacement metric} if there exists a map $f \colon X \to Y$ such that the induced map between the sets of connected components $\Bar{f} \colon \pi_{0}(X) \to \pi_{0}(Y)$ is a bijection and the restriction map $f|_{O} \colon O \to \Bar{f}(O)$ to any connected component $O$ of $X$ is a quasi-isometry with respect to the displacement metrics.

\begin{cor}\label{cor:all_conn_cmpts_homqdle_qi_dis}
    Let $X$ be a quandle.
    Suppose that the displacement group $\Dis(X)$ is finitely generated. 
    If the quandle $X$ is homogeneous, then any two connected components $O$ and $O^{\prime}$ with displacement metrics are quasi-isometric.
\end{cor}
\begin{proof}
    The claim follows by applying \cref{prop:conn_cmpts_sch_graph_dis}, \cref{prop:properties_displ_grp} (1) and (4), and by replacing $\Inn(X)$ with $\Dis(X)$ in the proof of \cref{cor:all_conn_cmpts_homqdle_qi_inn}, which yields a similar argument.
\end{proof}

At the end of this subsection, we consider the case that the displacement group acts freely.
In this situation, the metric on each connected component of a quandle with the displacement metric can be identified with the word metric of the displacement group.
We remark that the inner automorphism group cannot act freely due to the first axiom of quandles.
\begin{prop}
  Let $X$ be a quandle.
  We suppose that $S$ is a finite generating set of the displacement group of $X$.
  Let us take a connected component $O \subset X$ and a base point $p \in O$.
  If the displacement group acts freely on $O$,
  then the map $\alpha\colon (\Dis(X), d_S) \to (O, d^\Dis_S)$ defined by $\alpha(g) = p \cdot g$ is an isometry, where $d_{S}$ is the word metric of $\Dis(X)$ with respect to $S$. 
\end{prop}
\begin{proof}
  Since the action of the displacement group $\Dis(X)$ on a connected component is transitive by \cref{prop:properties_displ_grp} (4), we have that the map $\alpha$ is surjective.
  It remains to show that $\alpha$ preserves the metric.
  Let us take $x, y \in \Dis(X)$ with $d_S(x, y) = n$.
 Thus, there exist $a_1, \dots, a_n \in S \cup S^{-1}$ such that 
  $y = x a_1 \cdots a_n$.
  Then, we have $\alpha(y) = \alpha(x) \cdot (a_1 \cdots a_n)$, and hence it satisfies $d_S^\Dis(\alpha(x), \alpha(y)) \leq n$.
  Here, let us assume that $d_S^\Dis(\alpha(x), \alpha(y)) = m < n$.
  Then there exists $b_1, \dots, b_m \in S \cup S^{-1}$ such that $\alpha(y) = \alpha(x) \cdot (b_1 \cdots b_m)$.
  Since the action of $\Dis(X)$ on $O$ is free, it satisfies $a_1 \cdots a_n = b_1 \cdots b_m$.
 Therefore, we have $y = x b_1 \cdots b_m$, and hence $d_S(x, y) \leq m$.
 Then we conclude $d_S(x, y) = n > m \geq d_S(x, y)$, which is a contradiction.
  This completes the proof.
\end{proof}

By the above proposition,
we immediately obtain the following theorem.
\begin{thm}\label{quandle-theoretic Milnor-Svarz}
  Let $X$ be a quandle with a connected component $O$ on which the displacement group $\Dis(X)$ freely acts.
  If the group $\Dis(X)$ is finitely generated, then the metric space $O$ with a displacement metric is quasi-isometric to the metric space $\Dis(X)$ with a word metric.
\end{thm}

\begin{rem}
  Eisermann \cite{Eisermann-2014-QuandleCoveringsTheirGalois} developed the covering theory of quandles.
  In particular, he defined the simply connectedness for quandles.
  The displacement group of such a quandle acts freely.
  In other words, by \cref{quandle-theoretic Milnor-Svarz}, a simply connected quandle $X$ is quasi-isometric to a finitely generated group acting (freely) on itself.
  Here, the theorem reminds us the Milnor--\v{S}varc lemma (for instance, see \cite{Loh-2017-GeometricGroupTheoryIntroduction}).
  The lemma states that the universal covering manifold $\widetilde{M}$ of a compact Riemannian manifold $M$ is quasi-isometric to the fundamental group of $M$, which is finitely generated and acts properly discontinuously on the simply connected space $\widetilde{M}$.
  Hence, \cref{quandle-theoretic Milnor-Svarz} may be regarded as an analogue of the Milnor--\v{S}varc lemma in quandle theory.
\end{rem}

\subsection{Difference of inner automorphism groups and displacement groups}\label{sec:ex_not_qi_inn_dis}
As discussed above, we have defined two metrics on a quandle by natural group actions.
If both groups are finitely generated, then each connected component of the quandle admits two distinct quasi-isometry classes, corresponding to the metrics induced by these two group actions.
In this subsection, we give an explicit quandle whose quasi-isometry classes, as induced by the two group actions, are not the same.
More precisely, the infinite dihedral quandle $R_\infty = (\mathbb{Z}, \qop)$ given in \cref{ex:infty_dihedral_qdle} is the one.
Here we recall that the quandle $R_\infty$ is homogeneous, and it has two connected components $O_{\mathrm{even}}$ and $O_{\mathrm{odd}}$.
It is enough to focus only on the structure of $O_{\mathrm{even}}$ by \cref{cor:all_conn_cmpts_homqdle_qi_inn,cor:all_conn_cmpts_homqdle_qi_dis}. 
We now compute a quasi-isometry invariant, the number of ends (\cref{def:number_ends_graph}), for the set $O_{\mathrm{even}}$ with metrics defined by the group actions.
First, we consider the case of the inner metric. 
We also recall that the set $A = \{s_{0}, s_{1}\}$ is a generating set of $\Inn(R_{\infty})$.

\begin{lem}\label{lem:number_of_ends_inn}
    The number of ends for the metric space $(O_{\mathrm{even}}, \dInn_{A})$ is equal to one.
\end{lem}

\begin{proof} 
    Here, we denote the set $O_{\mathrm{even}}$ by $O$.
    We define a map $\gamma \colon \mathbb{Z}_{\geq 0} \to O$ by 
    \begin{align*}
        \gamma(k) = 
        \begin{cases}
            -k      & \textit{if $k$ is even}, \\
            k+1     & \textit{if $k$ is odd}. \\
        \end{cases}
    \end{align*}
    Then the map $\gamma$ is bijective and satisfies
    \begin{align*}
        \gamma(k) \cdot s_{0} = 
        \begin{cases}
            \gamma(0)   & \textit{if $k=0$}, \\
            \gamma(k-1) & \textit{if $k$ is even and positive}, \\
            \gamma(k+1) & \textit{if $k$ is odd}, \\
        \end{cases}
        \quad
        \gamma(k) \cdot s_{1} = 
        \begin{cases}
            \gamma(k+1)    & \textit{if $k$ is even}, \\
            \gamma(k-1)    & \textit{if $k$ is odd}. \\
        \end{cases}
    \end{align*}
    We equip the set $\mathbb{Z}_{\geq 0}$ with the standard graph metric $d_{\mathbb{Z}{\geq 0}}$, where $k$ and $k+1$ are joined by an edge of length one. 
    Observe that the map $\gamma \colon (\mathbb{Z}_{\geq 0}, d_{\mathbb{Z}_{\geq 0}}) \to (O, \dInn)$ is an isometry.
    Since the number of ends is invariant under quasi-isometries, we obtain 
    \begin{align*} 
        e(O, \dInn_A) = e(\mathbb{Z}_{\geq 0}, d_{\mathbb{Z}_{\geq 0}}).
    \end{align*}
    It is easy to see that $\mathbb{Z}_{\geq 0}$ is one-ended, so we conclude as desired.
\end{proof}

Next, we turn to the displacement metric. 
We recall that the set $U = \{s_{1}s_{0}^{-1}\}$ is a generating set of $\Dis(R_{\infty})$.
\begin{lem}\label{lem:number_of_ends_dis}
    The number of ends for the metric space $(O_{\mathrm{even}}, \dDis_{U})$ is equal to two.
\end{lem}

\begin{proof}
Here, we also denote the set $O_{\mathrm{even}}$ by $O$.
We equip the set $\mathbb{Z}$ with the natural metric $d_{\mathbb{Z}}$, where $k$ and $k \pm 1$ are joined by an edge of length one.
We define a map $\gamma^{\prime} \colon \mathbb{Z} \to O$ by
\begin{align*}
    \gamma^{\prime}(k) = 2k,\quad \text{for all } k \in \mathbb{Z}.
\end{align*}
Then the map $\gamma^{\prime}$ is bijective and satisfies
\begin{align*}
    \gamma^{\prime}(k) \cdot (s_{1}s_{0}^{-1}) = \gamma^{\prime}(k-1),\quad
    \gamma^{\prime}(k) \cdot (s_{1}s_{0}^{-1})^{-1} = \gamma^{\prime}(k+1).
\end{align*}
Hence, the map $\gamma^{\prime} \colon (\mathbb{Z}, d_{\mathbb{Z}}) \to (O, \dDis_U)$ is an isometry, and we have
\begin{align*} 
    e(O, \dDis_U) = e(\mathbb{Z}, d_{\mathbb{Z}}).
\end{align*}
Since $\mathbb{Z}$ is two-ended, we conclude as desired.
\end{proof}

The next theorem follows from the above two lemmas.

\begin{thm}\label{thm:fgqdle_sch_graph_inn_dis_notqi}
    There exist a quandle $X$ and a connected component $O \subset X$ which satisfy the following properties:
    \begin{enumerate}[font=\normalfont]
        \item The groups $\Inn(X)$ and $\Dis(X)$ are finitely generated.
        \item For any finite generating set $A$ of $\Inn(X)$ and $U$ of $\Dis(X)$, the metric spaces $(O, \dInn_A)$ and $(O, \dDis_{U})$ are not quasi-isometric.
    \end{enumerate}
\end{thm}

\begin{proof}
Let $X = R_\infty$ be the infinite dihedral quandle, as defined in \cref{ex:infty_dihedral_qdle}. 
As shown in \cref{ex:infty_dihedral_qdle}, both $\Inn(R_\infty)$ and $\Dis(R_\infty)$ are finitely generated. 
We fix generating sets $A = \{s_0, s_1\} \subset \Inn(R_\infty)$ and $U = \{s_0 s_1\} \subset \Dis(R_\infty)$.
By \cref{lem:number_of_ends_inn}, the number of ends for each connected component for the Schreier graph $\SchgraphInn_A(R_\infty)$ is one. 
In contrast, by \cref{lem:number_of_ends_dis}, the number of ends for each component of $\SchgraphDis_U(R_\infty)$ is two.
Since the number of ends is a quasi-isometry invariant, and since the components $(X, \dInn_A)$ and $(X, \dDis_U)$ have different numbers of ends, it follows that these components are not quasi-isometric.
This shows that $(X, \dInn_{A})$ and $(X, \dDis_{U})$ are not quasi-isometric.
\end{proof}
\section{Displacement groups of generalized Alexander quandles}\label{sec:generalized Alexander quandle}
In \cref{sec:examples}, we provide some examples of quandles that are quasi-isometric to typical metric spaces.
Most of these are obtained by applying \cref{quandle-theoretic Milnor-Svarz}.
In this section, we study the case where the assumption of the theorem holds, that is, the displacement group acts freely.
Quandles for which these conditions hold are essentially isomorphic to generalized Alexander quandles as shown in \cref{quandle with transitive action}.
These quandles play an important role in quandle theory.
For example, any homogeneous quandle is represented as a quotient of a generalized Alexander quandle 
(see \cite{Joyce-1982-ClassifyingInvariantKnotsKnot,Hulpke-2016-ConnectedQuandlesTransitiveGroups}).
These quandles are studied in detail in 
\cite{Higashitani-2024-GeneralizedAlexanderQuandlesFinite,Higashitani-2024-ClassificationGeneralizedAlexanderQuandles}.
First, we recall the definition of generalized Alexander quandles.

\begin{dfn}
  Let $G$ be a group and let $\sigma\colon G \to G$ be its automorphism.
  The \emph{generalized Alexander quandle} is a group equipped with a binary operation $\qop$ given by
  \[
    x \qop y \coloneqq \sigma(x y^{-1})y
  \]
  for $x, y \in G$, and is denoted by $\GAlex(G, \sigma)$.
\end{dfn}

One can easily check that the generalized Alexander quandle is a quandle.
By induction, we have $x \qop^n y = \sigma^{n}(xy^{-1})y$ for any $x, y \in \GAlex(G, \sigma)$ and $n \in \ZZ$.
The group $G$ acts on $\GAlex(G, \sigma)$ from the right as quandle automorphisms by $x \cdot g \coloneqq xg$ 
for $x \in \GAlex(G, \sigma)$ and $g \in G$.
In other words, the map $R_g\colon G \to G$ defined by $R_g(x) = xg$ for $g \in G$ is in $\Aut(\GAlex(G, \sigma))$.
In the following, we identify the group $G$ as a subgroup of $\Aut(\GAlex(G, \sigma))$.
In particular, a generalized Alexander quandle is a homogeneous quandle.
Conversely, we now show that a quandle for which a normal subgroup of its automorphism group acts freely and transitively is isomorphic to a generalized Alexander quandle.

\begin{prop}\label{quandle with transitive action}
  Let $X$ be a quandle.
  If there exists a normal subgroup $G$ of $\Aut(X)$ that acts freely and transitively on $X$, then the quandle $X$ is isomorphic to a generalized Alexander quandle as quandles.
\end{prop}
\begin{proof}
  First, we assume that a normal subgroup $G$ of $\Aut(X)$ acts freely and transitively on $X$. 
  Let us take a base point $x_0 \in X$.
  Since $G$ is a normal subgroup, it is closed under the conjugation by any element in $\Aut(X)$.
  Thus, the map $\sigma\colon G \to G$ defined by $\sigma(g) \coloneqq s_{x_0}^{-1} g s_{x_0}$ is a well-defined group automorphism on $G$.
  Here, we now show that the map $f\colon \GAlex(G, \sigma) \to X$ defined by $f(g) \coloneqq x_0 \cdot g$ is a quandle isomorphism.
  The map $f$ is bijective since the action of $G$ on $X$ is transitive and free.
  For $g, h \in \GAlex(G, \sigma)$, we have 
  \begin{align*}
    f(g \qop h) &= f(\sigma(gh^{-1})h) 
    = f(s_{x_0}^{-1}gh^{-1}s_{x_0} h)
    = x_0 \cdot (s_{x_0}^{-1}gh^{-1}s_{x_0} h)
    = (x_0 \cdot s_{x_0}^{-1}) \cdot (g \, (h^{-1}s_{x_0} h))\\
    &= x_0 \cdot (gs_{x_0 \cdot h})
    = (x_0 \cdot g) \cdot s_{x_0 \cdot h}
    = (x_0 \cdot g) \qop (x_0 \cdot h)
    = f(g) \qop f(h),
  \end{align*}
  where we use the first axiom of quandles and the definition of quandle homomorphisms in the fifth equality.
  Therefore, the map $f$ is a quandle isomorphism, as desired.

\end{proof}

\begin{rem}
  In general, the group $G$ is not a normal subgroup in the automorphism group of $\GAlex(G, \sigma)$.
  In fact, consider the case $G = \ZZ/n\ZZ$ for $n > 2$ and $\sigma = \mathrm{id}$.
  Then, the quandle $X = \GAlex(G, \sigma)$ is trivial, and its automorphism group is isomorphic to the symmetric group $\mathfrak{S}_n$.
\end{rem}

\begin{rem}
  It is known that a group object in the category of quandles must be isomorphic to a generalized Alexander quandle \cite{Andruskiewitsch-2003-RacksPointedHopfAlgebras}, but the converse is not true in general.
  In fact, for a generalized Alexander quandle $\GAlex(G, \sigma)$, the left action defined by $g \cdot x = gx$ for $g \in G$ and $x \in \GAlex(G, \sigma)$ is a quandle isomorphism if and only if the element $\sigma(g)^{-1} g$ is in the center of $G$ for any $g$.
\end{rem}

We denote the connected component of the identity $1 \in \GAlex(G, \sigma)$ by $P$.
It is known that the subset $P$ is a subquandle of $\GAlex(G, \sigma)$ and a normal subgroup of $G$ (see {\cite[Proposition 3.1]{Higashitani-2024-ClassificationGeneralizedAlexanderQuandles}}).
The following lemma was given in \cite{Higashitani-2024-GeneralizedAlexanderQuandlesFinite} for finite generalized Alexander quandles, but it holds for the general case.

\begin{lem}[cf.~~{\cite[Lemma 3.1]{Higashitani-2024-GeneralizedAlexanderQuandlesFinite}}]\label{identity component is Dis}
  Let $x \in P$.
  Then the map $R_x\colon G \to G$ defined by $R_x(y) = yx$ is contained in the displacement group $\Dis(\GAlex(G, \sigma))$.
\end{lem}
\begin{proof}
  Since $x \in P$ and by \cref{prop:properties_displ_grp} (4), there exists $g \in \Dis(\GAlex(G, \sigma))$ such that $x = 1 \cdot g$.
  By \cref{prop:properties_displ_grp} $(3)$, there exist $a_1, \dots, a_n \in \GAlex(G, \sigma)$ and $k_1, \dots, k_n \in \ZZ$ with $\sum_{i=1}^n k_i = 0$ such that $g = s_{a_1}^{k_1} \cdots s_{a_n}^{k_n}$.
  Thus, we have
  \begin{align*}
    x &= 1 \cdot (s_{a_1}^{k_1} \cdots s_{a_n}^{k_n})
    = 1 \qop^{k_1} a_1 \qop^{k_2} \cdots \qop^{k_n} a_n\\
    &= \sigma^{k_n}(\sigma^{k_{n-1}}(\cdots (\sigma^{k_2}(\sigma^{k_1}(1 a_1^{-1})a_1a_2^{-1})a_2a_3^{-1}) \cdots) a_{n-1}a_{n}^{-1})a_n.
  \end{align*}
  Hence, any element $y \in \GAlex(G, \sigma)$ satisfies that 
  \begin{align*}
    R_x(y) &= y \sigma^{k_n}(\sigma^{k_{n-1}}(\cdots (\sigma^{k_2}(\sigma^{k_1}(1 a_1^{-1})a_1a_2^{-1})a_2a_3^{-1}) \cdots) a_{n-1}a_{n}^{-1})a_n\\
    &= \sigma^{k_n+\cdots+k_1}(y) \sigma^{k_n}(\sigma^{k_{n-1}}(\cdots (\sigma^{k_2}(\sigma^{k_1}(1 a_1^{-1})a_1a_2^{-1})a_2a_3^{-1}) \cdots) a_{n-1}a_{n}^{-1})a_n\\
    &= \sigma^{k_n}(\sigma^{k_{n-1}}(\cdots (\sigma^{k_2}(\sigma^{k_1}(y a_1^{-1})a_1a_2^{-1})a_2a_3^{-1}) \cdots) a_{n-1}a_{n}^{-1})a_n\\
    &= y \cdot (s_{a_1}^{k_1} \cdots s_{a_n}^{k_n}) \\
    &= y \cdot g.
  \end{align*}
  Therefore, we have $R_x = g \in \Dis(\GAlex(G, \sigma))$, as desired.
\end{proof}

We now show that the subgroup $P$ is isomorphic to the displacement group.

\begin{prop}[cf.~~\cite{Higashitani-2024-GeneralizedAlexanderQuandlesFinite,Higashitani-2024-ClassificationGeneralizedAlexanderQuandles}]\label{PisDis}
  The displacement group of $X \coloneqq \GAlex(G, \sigma)$ is isomorphic to the group $P$ as groups.
\end{prop}
\begin{proof}
  Let us define a map $R\colon P \to \Aut(X)$ by $R(x) \coloneqq R_x$.
  It is clear that the map $R$ is an injective group homomorphism.
  By \cref{identity component is Dis}, the image of $R$ is included in $\Dis(X)$.
  We now show the inverse inclusion $\mathrm{Im}(R) \supset \Dis(X)$.
  Since the group $\Dis(X)$ is generated by the set $\{s_x s_y^{-1} \mid x, y \in X\}$, it is enough to show that a generator $s_x s_y^{-1}$ of $\Dis(X)$ is in the image of $R$ for any $x, y \in X$.
  Let us put $g \coloneqq 1 \qop x \qop^{-1} y \in P$.
  Then any $z \in X$ satisfies
  \begin{align*}
    R_g(z) = zg 
    = z (1 \qop x \qop^{-1} y) 
    = z x^{-1} \sigma^{-1}(x y^{-1})y
    = \sigma^{-1} (\sigma(zx^{-1})x y^{-1})y
    = z \qop x \qop^{-1} y
    = z \cdot (s_xs_y^{-1}).
  \end{align*}
  Hence, we have $R_g = s_xs_y^{-1}$, which completes the proof.
\end{proof}

As a corollary of \cref{PisDis},
we calculate the displacement group for 
the case that the group automorphism is given as 
an inner automorphism.
For a group $G$ and its subgroup $H$,
we denote the subgroup generated by $\{[g, h] \mid g \in G, h \in H\}$ by $[G, H]$.
We remark that if the subgroup $H$ is normal in $G$, then $[G, H]$ is a normal subgroup of $H$.
\begin{cor}\label{Dis of GAlex Inn}
  If the group automorphism $\sigma$ is an inner automorphism of an element $g \in G$, that is $\sigma(x)\coloneqq g^{-1}xg$, then $\Dis(\GAlex(G, \sigma))$ is isomorphic to the subgroup $[G, \langle\langle g \rangle\rangle_G]$,
  where $\langle\langle g \rangle\rangle_G$ is the normal closure of $g$ in $G$.
\end{cor}
\begin{proof}
  We show that the normal subgroup $P$ in $G$ is equal to $[G, \langle\langle g \rangle\rangle_G]$.
  Let us take $x \in P$.
  Then there exists $g \in \Dis(X)$ such that $x = 1 \cdot g$.
  By an argument similar to the proof of \cref{identity component is Dis}, there exist $a_1, \dots, a_n \in X$ and $k_1, \dots, k_n \in \ZZ$ with $\sum_{i=1}^n k_i = 0$ such that $g = s_{a_1}^{k_1} \cdots s_{a_n}^{k_n}$, and we have
  \begin{align*}
    x &= \sigma^{k_n}(\sigma^{k_{n-1}}(\cdots (\sigma^{k_2}(\sigma^{k_1}(1 a_1^{-1})a_1a_2^{-1})a_2a_3^{-1}) \cdots) a_{n-1}a_{n}^{-1})a_n\\
    &= g^{-k_n} (\cdots (g^{-k_2} (g^{-k_1} a_1^{-1} g^{k_1}) a_1 a_2^{-1} g^{k_2}) a_2a_3^{-1} \cdots a_{n-1} a_n^{-1}) g^{k_n} a_n \\
    &= g^{-(k_n+ \cdots +k_1)} a_1^{-1} g^{k_1} a_1 a_2^{-1} g^{k_2} a_2 \cdots a_n^{-1}g^{k_n} a_n \\
    &= a_1^{-1} g^{k_1} a_1 a_2^{-1} g^{k_2} a_2 \cdots a_n^{-1}g^{k_n} a_n 
    \in \langle\langle g \rangle\rangle_G.
  \end{align*}
  In particular, the homomorphism $\langle\langle g \rangle\rangle_G \to \langle\langle g \rangle\rangle_G/[G, \langle\langle g \rangle\rangle_G]$ carries $x$ to $0$ since $\sum_{i=1}^n k_i = 0$.
  Therefore we have $x \in [G, \langle\langle g \rangle\rangle_G]$, and hence $P \subset [G, \langle\langle g \rangle\rangle_G]$.

  Conversely, we now show that $[G, \langle\langle g \rangle\rangle_G] \subset P$.
  Let us take $x \in G$ and $y \in \langle\langle g \rangle\rangle_G$.
  Then there exist $a_i \in G$ and $\varepsilon_i \in \{\pm 1\}$ such that $y = a_1^{-1} g^{\varepsilon_1} a_1 \cdots a_n^{-1} g^{\varepsilon_n} a_n$.
  Here, let us take elements $\alpha \coloneqq s_{a_1}^{\varepsilon_1} \cdots s_{a_n}^{\varepsilon_n}$.
  Then, we have $y = e \cdot \alpha \in P$.
  Since the subgroup $P$ is normal in $G$, 
  the element $[x, y] = (xyx^{-1})y$ is in $P$.
  Therefore, we obtain $[G, \langle\langle g \rangle\rangle_G] \subset P$, which completes the proof.
\end{proof}


As a conclusion to this section, quasi-isometry classes of connected components of generalized Alexander quandles are determined.
\begin{thm}\label{thm:qi_GAlex}
  Let $G$ be a group and let $\sigma$ be its group automorphism.
  If the displacement group of the generalized Alexander quandle $X \coloneqq \GAlex(G, \sigma)$ is finitely generated, then any connected component $O$ of the quandle with a displacement metric is quasi-isometric to the displacement group with a word metric. 
\end{thm}
\begin{proof}
  Since a generalized Alexander quandle is homogeneous, we can apply \cref{cor:all_conn_cmpts_homqdle_qi_dis}.
  Thus, it is enough to show that the connected component $P$ of the identity $1 \in G$ is quasi-isometric to the displacement group.
  We recall that the set $P$ is a subgroup in $G$, and the map $R\colon P \to \Dis(X)$ is a group isomorphism by \cref{PisDis}.
  Here, the action of $\Dis(X)$ is free since it is given by the action of the subgroup $P$ in $G$ through the isomorphism $R$.
  Therefore, by applying \cref{quandle-theoretic Milnor-Svarz}, we obtain the assertion.
\end{proof}

\section{Examples}\label{sec:examples}
In this section, we give some examples of quandles quasi-isometric to certain metric spaces, the trees, the Euclidean spaces, the hyperbolic plane, and 3-dimensional homogeneous spaces.

First, we consider the inner metric.
Recall the free quandle defined as follows (see \cite[section 8.6]{Kamada-2017-SurfaceKnots4Space} for details):
let $A$ be a set and let $F[A]$ be the free group on $A$.
We denote the direct product of $A$ and $F[A]$ by $FR[A]$.
Define an equivalence relation on $FR[A]$ as follows: 
\[
  (a, w) \sim (b, v) \quad \text{if} \quad
  a = b, \quad \text{and} \quad w = a^n v \in F[A] \quad
  \text{for some } n \in \ZZ,
\]
We write the quotient set $FR[A]/{\sim}$ as $FQ[A]$.
We denote an element in $FQ[A]$ by $a^w \coloneqq [(a, w)]$.
The free quandle on $A$ is the set $FQ[A]$ equipped with the quandle structure $\qop$ given by $a^w \qop b^v = a^{wv^{-1}bv}$.
Note that the free quandle $FQ[A]$ is 
generated by $A = \{a^1 \in FQ[A] \mid a \in A\} \subset FQ[A]$.
We now identify $A = \{s_a \mid a \in A\} \subset \Inn(FQ[A])$.
Like the Cayley graph of the free group, the inner graph of the free quandle is quasi-isometric to a tree.
\begin{prop}
  Let $A$ be a finite set with the cardinality greater than one, and let $FQ[A]$ be the free quandle generated by $A$.
  Then, each connected component of $FQ[A]$ with the inner metric $d^\Inn_A$ is quasi-isometric to a tree.
\end{prop}

\begin{proof}
  It is enough to show that if $\gamma$ is a simple loop in the Schreier graph, then its length is at most $1$.
  Let us take a point $a^w \in FQ[A]$ and a simple loop $\gamma$ at $a^w \in FQ[A]$ with length $k$.
  Then the loop $\gamma$ is denoted as a sequence $\{a_{i_1}, \dots, a_{i_k}\}$ of labels in $A$, and we have
  \[
    a^w  = a^w \cdot (s_{a_{i_1}}^{\varepsilon_1} \cdots s_{a_{i_k}}^{\varepsilon_k})
    = a^{w a_{i_1}^{\varepsilon_1} \cdots a_{i_k}^{\varepsilon_k}},
  \]
  where $\varepsilon_j \in \{\pm 1\}$.
  Here, the word $c \coloneqq a_{i_1}^{\varepsilon_1} \cdots a_{i_k}^{\varepsilon_k}$ is reduced in $F[A]$ since the loop $\gamma$ is simple.
  By the definition of the free quandle, there exists $n \in \ZZ$ such that 
  \[
    w = a^n w c \in F[A].
  \]
  Hence, we have $c = w^{-1} a^{-n} w$.
  If the element $w$ is the identity, then it satisfies that $a^{-n} = c$.
  In this case, the loop $\gamma$ consists of $k$ edges labeled by $a$, and hence $k$ is equal to zero or one since $\gamma$ is simple.
  Next, we assume that the element $w$ is not the identity.
  Then the element $w$ is uniquely presented by the reduced word $\alpha_1^{\delta_1} \cdots \alpha_m^{\delta_m}$ with $m > 0$, $\alpha_i \in A$, and $\delta_i \in \{\pm 1\}$.
  By using the equivalence relation $\sim$ and replacing $w$, we can assume that $\alpha_1$ is not equal to $a$.
  Then, the word $\alpha_m^{-\delta_m} \cdots \alpha_1^{-\delta_1} a^{-n} \alpha_1^{\delta_1} \cdots \alpha_m^{\delta_m}$ is reduced and presents the element $c$ in $F[A]$.
  We now consider that one walks along $\gamma$.
  Firstly, it takes us from the vertex $a^w$ to the vertex $a^1$ by walking the edges labeled by the word with respect to $w^{-1}$.  
  Secondly, we walk the edges labeled by $a^{-n}$, but this is a loop at $a$.
  Hence we have that $n=0$ since $\gamma$ is simple.
  Therefore, it satisfies that 
  \[
    c = w^{-1} a^{-n} w = 1.
  \]
  Therefore we have $k = 0$, which completes the proof.
\end{proof}

\begin{rem}
  We note that the inner automorphism group of $FQ[A]$ is isomorphic to the free group $F[A]$.
  The displacement group of $FQ[A]$ is isomorphic to the subgroup $D$ of $FQ[A]$ generated by $\{g^{-1}ag h^{-1}a'^{-1}h \mid g,h \in F[A], a, a' \in A\}$.
  This group is not finitely generated.
  Thus, one cannot apply \cref{thm:qi_indep_ch_genset_sch_graph_dis} to finitely generated free quandles.
\end{rem}

In the rest of this paper, 
we consider displacement metrics for generalized Alexander quandles of some discrete groups.
First, we give a quandle whose connected components are quasi-isometric to the Euclidean space.
\begin{prop}
  Let $t$ be a group automorphism of $\ZZ^n$, and let $X = \GAlex(\ZZ^n, t)$.
  Then the following hold:
  \begin{enumerate}[font=\normalfont]
    \item 
    The displacement group $\Dis(X)$ is isomorphic to $(1-t^{-1}) \ZZ^n$.

    \item
    Any connected component of $X$ with a displacement metric is quasi-isometric to the $k$-dimensional Euclidean space, where $k = \rank(1-t^{-1})$.
  \end{enumerate}
\end{prop}

\begin{proof}
  For any $x, y , z\in X$, the generator $s_xs_y^{-1}$ of $\Dis(X)$ satisfies that 
  \begin{align*}
    z \cdot (s_xs_y^{-1}) 
    = (z \qop x) \qop^{-1} y
    = t^{-1}(tz + (1-t)x) + (1-t^{-1})y
    = z + (1-t^{-1})(y-x).
  \end{align*}
  Hence, the displacement group $\Dis(X)$ consists of the parallel transformations of vectors in $(1-t^{-1}) \ZZ^n$.
  Therefore we have $(1)$.

  Since it satisfies that $(1-t^{-1}) \ZZ^n \cong \ZZ^k$, where $k = \rank(1-t^{-1})$, the displacement group is quasi-isometric to the $k$-dimensional Euclidean space.
  Therefore, we have $(2)$ by \cref{quandle-theoretic Milnor-Svarz}.
\end{proof}

Next, we consider quandles obtained from the subgroups consisting of orientation preserving elements in the triangle groups.
The connected components of these quandles are quasi-isometric to $2$-dimensional geometries with constant curvatures.
\begin{prop}\label{triangle group}
  Let $\Delta^+(p, q, r)$ be the index $2$-subgroup of the triangle group $\Delta(p,q,r)$, that is, 
  \[
    \Delta^+(p, q, r) = \langle a, b, c \mid a^p = b^q = c^r = abc = 1 \rangle_\Grp.
  \]
  Let us define a group automorphism $\sigma\colon \Delta^+(p, q, r) \to \Delta^+(p, q, r)$ by $\sigma(g) \coloneqq a^{-1} g a$.
  Then the generalized Alexander quandle $X \coloneqq \GAlex(\Delta^+(p, q, r), \sigma)$ satisfies the following:
  \begin{enumerate}[font=\normalfont]
    \item The displacement group $\Dis(X)$ is isomorphic to a finite index subgroup in  $\Delta^+(p,q,r)$.

    \item If $\frac{1}{p} + \frac{1}{q} + \frac{1}{r} > 1$, then the quandle $X$ is finite.

    \item If $\frac{1}{p} + \frac{1}{q} + \frac{1}{r} = 1$, then a connected component of $X$ with a displacement metric is quasi-isometric to the Euclidean plane.

    \item If $\frac{1}{p} + \frac{1}{q} + \frac{1}{r} < 1$, then a connected component of $X$ with a displacement metric is quasi-isometric to the $2$-dimensional hyperbolic space.
  \end{enumerate}
\end{prop}
\begin{proof}
  First, we show (1).
  We denote the normal closure $\langle\langle a \rangle\rangle_{\Delta^+(p,q,r)}$ of $a \in \Delta^+(p,q,r)$ by $A$.
  By \cref{Dis of GAlex Inn}, the displacement group $\Dis(X)$ is isomorphic to the subgroup $[\Delta^+(p,q,r), A]$, and we now show that this is a finite index subgroup of $\Delta^+(p,q,r)$.
  Since the subgroup includes the commutator subgroup $[A, A]$, it is enough to show that $[A, A]$ is a finite index subgroup in $\Delta^+(p,q,r)$.
  Since the order of $a$ is equal to $p$, the abelianization of the group $A$ is isomorphic to $\ZZ/p\ZZ$.
  In particular, the subgroup $[A, A]$ is finite index in $A$.
  Hence, the group $\Delta^+(p,q,r)/A$ is presented as 
  \[
    \Delta^+(p,q,r)/A = \langle a, b, c \mid a^p = b^q = c^r = abc = a = 1 \rangle_\Grp.
  \]
  Thus, the quotient group $\Delta^+(p,q,r)/A$ is isomorphic to the cyclic group with the order equal to the greatest common divisor of $q$ and $r$.
  Hence, the group $A$ is a finite index subgroup of $\Delta^+(p,q,r)$, as desired.

  By $(1)$, the displacement group $\Dis(X)$ is isomorphic to a finite index subgroup of $\Delta^+(p,q,r)$. 
  We recall that every finite index subgroup of a finitely generated group is finitely generated, and is quasi-isometric to the original group (for instance, see \cite{Loh-2017-GeometricGroupTheoryIntroduction}).
  Hence, the displacement group $\Dis(X)$ is finitely generated and is quasi-isometric to $\Delta^+(p,q,r)$.
  Thus, any connected component of $X$ is quasi-isometric to $\Delta^+(p,q,r)$.
  By Poincaré's fundamental polyhedron theorem (for instance, see \cite{Ratcliffe-2019-FoundationsHyperbolicManifolds}), the group $\Delta^+(p, q, r)$ acts properly discontinuously and cocompactly on the $2$-dimensional sphere $S^2$ if $\frac{1}{p} + \frac{1}{q} + \frac{1}{r} > 1$, (resp. the $2$-dimensional Euclidean space $\mathbb{E}^2$ if $\frac{1}{p} + \frac{1}{q} + \frac{1}{r} = 1$, or the $2$-dimensional hyperbolic space $\mathbb{H}^2$ if $\frac{1}{p} + \frac{1}{q} + \frac{1}{r} < 1$).
  By applying the Milnor-\v{S}varz lemma \cite{Loh-2017-GeometricGroupTheoryIntroduction}, we have assertions (2), (3) and (4), which complete the proof.
\end{proof}

\begin{rem}
  We use the notations in \cref{triangle group}. 
  Let $Y$ be the conjugation quandle consisting of the conjugacy class including $a$.
  Then, there exists a surjective quandle homomorphism $X \to Y$. 
  Moreover, the quandle $Y$ is a discrete subquandle of the corresponding space with some natural quandle structure. 
  For example, if $\frac{1}{p} + \frac{1}{q} + \frac{1}{r} < 1$, $Y$ is a discrete subquandle of $\mathbb{H}^2$ with a quandle structure given as rotations.
\end{rem}

\begin{proof}
  First, one can check directly that the map $\pi \colon X=\GAlex(\Delta^+(p,q,r), \sigma) \to Y$ defined by $\pi(g) \coloneqq g^{-1} a g$ is a surjective quandle homomorphism.
  Here, we show that $Y$ is a discrete subquandle of $\mathbb{H}^2$ if $\frac{1}{p} + \frac{1}{q} + \frac{1}{r} < 1$.
  For $\theta \in \RR$, we denote the $\theta$-rotation centered at $y \in \mathbb{H}^2$ by $\rho_y$.
  This gives a quandle structure of $\mathbb{H}^2$, that is, the binary operation $\qop^\theta$ on $\mathbb{H}^2$ defined by $x \qop^\theta y \coloneqq \rho_y^\theta(x)$ is a quandle structure on $\mathbb{H}^2$.
  Let us consider the properly discontinuous action of $\Delta^+(p,q,r)$ on $\mathbb{H}^2$. 
  Then, the generator $a$ acts as a $\frac{2\pi}{p}$-rotation.
  Since $Y$ is the conjugacy class of $a$, any element $y \in Y$ also acts as a $\frac{2\pi}{p}$-rotation at some point $c_y \in \mathbb{H}^2$.
  Here, one can easily check that the map $c: Y \to \mathbb{H}^2$ is an injective quandle homomorphism.
  Moreover, its image is discrete in $\mathbb{H}^2$ since $\Delta^+(p,q,r)$ acts properly and discontinuously on $\mathbb{H}^2$.
\end{proof}

For a knot $K$ in $3$-sphere $S^3$ and a positive integer $n$, a $3$-orbifold of the base space $S^3$ with the singular set $K$ whose cone-angle is equal to $\frac{2 \pi}{n}$ is denoted by $\mathcal{O}(K,n)$.
Let $G(K) = \pi_1(S^3 \setminus K)$ 
be the knot group of $K$,
and fix a meridian $\mu \in G(K)$.
Then,
it is known that the orbifold fundamental group 
$G_n(K) \coloneqq \pi_1^\mathrm{orb}(\mathcal{O}(K,n))$
is isomorphic to $G(K)/\langle\langle \mu^n \rangle\rangle_{G(K)}$,
where $\langle\langle \mu^n \rangle\rangle_{G(K)}$
is the normal closure of $\mu^n$ in $G(K)$.
We denote the image of the meridian in $G_n(K)$ 
by the same symbol $\mu$.
The orbifold is called \emph{geometric}
if the orbifold admits
one of the following geometric structures
\cite{Scott-1983-Geometries3Manifolds,Dunbar-1988-GeometricOrbifolds}:
\[
  S^2 \times \mathbb{E}^1,\, 
  S^3,\, 
  \mathbb{E}^3,\, 
  \mathrm{Nil},\,
  \mathbb{H}^2 \times \mathbb{E}^1,\,
  \widetilde{SL}_2,\,
  \mathrm{Sol},\,
  \mathbb{H}^3.
\]

\begin{prop}\label{3-orbifold group}
  Let $\sigma\colon G_n(K) \to G_n(K)$ be a group automorphism defined by $\sigma(g) = \mu^{-1} g \mu$.
  Then the quandle $X_n(K) \coloneqq \GAlex(G_n(K), \sigma)$ satisfies the following:
  \begin{enumerate}[font=\normalfont]
    \item The displacement group of $X_n(K)$ is isomorphic to the fundamental group $\pi_1(M_n(K))$, where $M_n(K)$ is the $n$-fold branched covering space along $K$.

    \item If the orbifold $\mathcal{O}(K,n)$ is geometric, then any connected component of the quandle $X_n(K)$ with a displacement metric is quasi-isometric to the universal covering space $\widetilde{\mathcal{O}(K,n)}$.
  \end{enumerate}
\end{prop}
\begin{proof}
  Since the knot group is normally generated  by $\mu$, the quotient group $G_n(K)$ is also normally generated by $\mu$.
  By \cref{Dis of GAlex Inn}, the displacement group $\Dis(X)$ is isomorphic to the commutator subgroup of the normal closure of $\mu$.
  Thus, the displacement group is isomorphic to the commutator subgroup of $G_n(K)$, which is just the fundamental group $\pi_1(M_n(K))$. 
  Hence, we have $(1)$. 
  
  By \cref{quandle with transitive action} and $(1)$, any connected component of the quandle $X_n$ with a displacement metric is quasi-isometric to the group $\pi_1(M_n(K))$ with a word metric.
  Since the orbifold $\mathcal{O}(K,n)$ is geometric, the universal covering space $\widetilde{\mathcal{O}(K,n)}$ has the geometric structure, and the group $\pi_1(M_n(K))$ acts isometrically and properly discontinuously.
  By the Milnor-\v{S}varc lemma \cite{Loh-2017-GeometricGroupTheoryIntroduction}, the group $\pi_1(M_n(K))$ is quasi-isometric to $\widetilde{\mathcal{O}(K,n)}$.
  Therefore, any connected component of the quandle $X_n(K)$ is quasi-isometric to $\widetilde{\mathcal{O}(K,n)}$, which completes the proof.
\end{proof}

\begin{rem}
  There exists a surjective quandle homomorphism from our quandle $X_n(K)$ to the knot $n$-quandle $Q_n(K)$.
  See \cite{Winker-1984-QUANDLESKNOTINVARIANTSNFOLD} for details.
\end{rem}

\begin{rem}
  \cref{3-orbifold group} gives many examples of quandles
  quasi-isometric to $3$-dimensional homogeneous spaces.
  \begin{enumerate}
    \item 
    Let $K$ be a hyperbolic knot. 
    If a positive integer $n$ is large enough, then the orbifold $\mathcal{O}(K,n)$ is hyperbolic by the hyperbolic Dehn surgery theorem
    (for example, see \cite{Purcell-2020-HyperbolicKnotTheory}).
    Hence, a connected component of $X_n(K)$ is quasi-isometric to $\mathbb{H}^3$.

    \item 
    Let $K$ be a Montesinos knot.
    It is known that the orbifold $\mathcal{O}(K, 2)$ has a Seifert structure by the Montesinos trick \cite[Proposition 12.31]{Burde-2003-Knots}.
    Hence, a connected component of $X_2(K)$ is quasi-isometric to one of the following geometries:
    $S^2 \times \mathbb{E}^1$, $S^3$, $\mathbb{E}^3$, $\mathrm{Nil}$, $\mathbb{H}^2 \times \mathbb{E}^1$ and $\widetilde{SL}_2$ (\cite{Scott-1983-Geometries3Manifolds,Dunbar-1988-GeometricOrbifolds}).
    For example, if $K$ is the $(-2,3,7)$-pretzel knot, then a connected component of $X_2(K)$ is quasi-isometric to $\widetilde{SL}_2$.

   \item 
    According to the classification of geometric orbifolds given by Dunbar \cite{Dunbar-1988-GeometricOrbifolds}, we can construct more examples.
    A connected component of the quandle $X_3(4_1)$ is quasi-isometric to $\mathbb{E}^3$, where $4_1$ is the figure-eight knot.
    A connected component of the quandle $X_6(3_1)$ is quasi-isometric to $\mathrm{Nil}$, where $3_1$ is the trefoil.
  \end{enumerate}
\end{rem}
\section*{Acknowledgements}
The authors are grateful to Prof.~Masato Mimura for his comment about the notion of the Schreier graph.
They would also like to thank Prof.~Hirotaka Akiyoshi, Prof.~Hiraku Nozawa, and Prof.~Hiroshi Tamaru for their helpful comments and encouragement.
The second author is supported by JST SPRING, Grant Number JPMJSP2139.
The third author is partially supported by JSPS KAKENHI Grant number 24K22836.

\bibliographystyle{spmpsci}
\bibliography{GeometricQuandle}

\end{document}